\documentclass[a4paper, 11pt]{amsart}
\usepackage{upgreek} 
\usepackage{hyperref,xy,doi}
\usepackage{diagbox}
\usepackage{soul}
\usepackage{enumerate}
\usepackage[top=3cm,bottom=3cm,left=2.5cm,right=2.5cm,footskip=17mm]{geometry}
\xyoption{all}
\usepackage{amsmath,amssymb, xy,wasysym,bm,amsthm}
\usepackage{tikz} 
\usepackage{xcolor}
\usepackage{subcaption}
\usetikzlibrary{cd}
\usepackage{color, comment}
\usepackage{multirow}

\numberwithin{equation}{section}
\numberwithin{figure}{section}

\theoremstyle{plain}
\newtheorem{theorem}{Theorem}[section]

\newtheorem{proposition}[theorem]{Proposition}
\newtheorem*{theorem*}{Theorem}
\newtheorem*{lemma*}{Lemma}
\newtheorem{lemma}[theorem]{Lemma}

\newtheorem{remark}[theorem]{Remark}

\newcommand{\Hom}{{\mathrm{Hom}}}

\newcommand{\tra}{{\mathrm{tra}}}

\newcommand{\Ho}{{\mathrm{H}}}

\newcommand{\wH}[1]{\widetilde{H_{#1}}}
\newcommand{\irr}{\mathrm{Irr}}

\begin{document}
\title{On twisted group ring isomorphism problem for \MakeLowercase{p}-groups }
\author{Gurleen Kaur}
\address{GK:  Indian Institute of Science Education and Research Mohali, Knowledge City, Sector 81, Mohali 140 306, India}
\email{gurleen@iisermohali.ac.in}
\author{Surinder Kaur}
\address{SK: Department of Mathematics and Statistics, Indian Institute of Technology Kanpur, Kanpur 208016, India}
\email{skaur.math@gmail.com}
\author{Pooja Singla}
\address{PS: Department of Mathematics and Statistics, Indian Institute of Technology Kanpur, Kanpur 208016, India}
\email{psingla@iitk.ac.in}
\date{}
{\maketitle}
\noindent
\begin{abstract}
  In this article, we explore the problem of determining isomorphisms between the twisted complex group algebras of finite $p$-groups. This problem bears similarity to the classical group algebra isomorphism problem and has been recently examined by Margolis-Schnabel. Our focus lies on a specific invariant, referred to as the generalized corank, which relates to the twisted complex group algebra isomorphism problem. We provide a solution for non-abelian $p$-groups with generalized corank at most three.
\end{abstract}

\section{Introduction}
The group ring $RG$, where $G$ is a finite group and $R$ is a commutative ring, holds significant importance in representation theory. Over the past few decades, there has been considerable interest in decoding information about a group $G$ from its group ring $RG$. One particularly challenging problem in this context is the isomorphism problem, which investigates whether a group ring uniquely determines its corresponding group. Specifically, if $RG$ and $RH$ are isomorphic as $R$-rings, does it imply that the groups $G$ and $H$ are isomorphic as well? For the current status of this problem, one can refer to \cite{MR142622, mr602901, MR4472590,margolis2018finite, sandling2006isomorphism}.
The solution to this problem depends mainly upon the ring under consideration. For example, all the finite abelian groups of a given order have isomorphic complex group algebras, whereas the rational group algebras of any two non-isomorphic abelian groups are always non-isomorphic (see \cite{MR34758}). In 1971, Dade~\cite{MR280610} constructed an example demonstrating the existence of two non-isomorphic metabelian groups that possess isomorphic group algebras over any field. Subsequently, Hertweck ~\cite{MR1847590} presented a counterexample to this phenomenon specifically for integral group rings, showcasing two non-isomorphic groups of even order whose integral group rings are isomorphic. However, the problem of determining whether integral group rings of groups with odd order exhibit isomorphism remains an open question. Additionally, investigating this problem in the context of modular group rings, in particular, for the group rings of finite $p$-groups over a field of characteristic $p$ has been of significant interest (see \cite{MR1896125}).

In recent times, a variant of the classical isomorphism problem known as the twisted group ring isomorphism problem  has gained considerable attention. The problem was initially introduced in \cite{MS} and has been further explored by the authors in \cite{MR4145799,margolis2020twisted}. In order to explain this version of the isomorphism problem, we start by introducing some notation.

Let $R$ be a commutative ring with unity and $R^\times$ be the unit group of $R$.
Following \cite{MR788161}, we denote the set of $2$-cocyles of $G$ by $ Z^{2}(G,R^{\times})$ and the second cohomology group of $G$ over $R^{\times}$ by $ H^{2}(G,R^{\times})$. For a $2$-cocycle $\alpha,$ let $[\alpha]$ denote its cohomology class. Given a ring $R$, we write $G\sim_{R}H$ if there exists an isomorphism $\psi:H^{2}(G,R^{\times}) \rightarrow H^{2}(H,R^{\times})$ such that $R^{\alpha}G \cong R^{\psi(\alpha)}H$ for every $[\alpha] \in H^{2}(G,R^{\times})$. The twisted group ring isomorphism problem is to determine the equivalence classes of  groups of order $n,$ under the relation $\sim_{R}$. We call these equivalence classes to be the \textbf{$R$-twist isomorphism classes}.

If two groups belong to the same $R$-twist isomorphism class, then their group rings over $R$ are isomorphic. The order of the group $\Ho^2(G, R^\times)$ remains unchanged under the $\sim_R$ relation for the twisted group ring of $G$. Throughout this article, our focus lies on the $\mathbb{C}$-twist isomorphism classes of finite $p$-groups. The structure of the complex group algebras remains invariant under $\mathbb{C}$-twist isomorphism. Also, the group $\Ho^2(G, \mathbb{C}^\times)$ is commonly referred to as the Schur multiplier of $G$.

In \cite[Theorem~4.3]{MS}, Margolis and Schnabel determined the $\mathbb{C}$-twist isomorphism classes of groups of order $p^4$, where $p$ is a prime. In the same article, they proved (see \cite[Lemma~1.2]{MS}) that any equivalence class of a finite abelian group  with respect to $\sim_\mathbb{C}$ is a singleton. Hence it is sufficient to focus on the classification of the $\mathbb C$-twist isomorphism classes of non-abelian finite groups. In this article, we continue this line of investigation of the $\mathbb {C}$-twist isomorphism classes of finite non-abelian $p$-groups by fixing the order of the Schur multiplier. In this direction, Green~\cite{Green} proved that the order of the Schur multiplier of a $p$-group $G$ of order $p^n$ is at most
$ p^{\frac{n(n-1)}{2}}.$
Niroomand~\cite{N2} improved this bound for non-abelian $p$-groups and proved that $|\Ho^2(G, \mathbb C^\times)| \leq p^{\frac{(n-1)(n-2)}{2}+1}$ for any non-abelian group $G$ of order $p^n$. Motivated by this result, a finite non-abelian $p$-group $G$ is said to have {\bf generalized corank} $s(G)$ if $|\Ho^2(G, \mathbb C^\times)| = p^{\frac{(n-1)(n-2)}{2}+1 - s(G)}.$

We study the $\mathbb{C}$-twist isomorphism classes of finite non-abelian $p$-groups by fixing their generalized corank. In particular, we describe the $\mathbb C$-twist isomorphism classes of all $p$-groups with $s(G) \leq 3$.  The classification of all non-isomorphic $p$-groups with $s(G) \leq 3$ is known in the literature by the work of P. Niroomand~\cite{N4} and S. Hatui~\cite{Hatui}. We use this classification along with the structure of the corresponding twisted group algebras to obtain our results. We use the following notation:

\begin{itemize}
  \item $C_{p^{n}}$ denotes the cyclic group of order $p^n$.
  \item $C_{p^{n}}^{m}$ or $C_{p^{n}}^{(m)}$ denote the direct product of $m$-copies of the cyclic group of order $p^n$.
  \item $H_{m}^{1}$ and $H_{m}^{2}$ denote the extraspecial $p$-groups of order $p^{2m+1}$ and of exponent $p$ and $p^2$, respectively.
  \item $H.K$ denotes the central product of the groups $H$ and $K$.
  \item $E(r) = E.C_{p^{r}}$, where $E$ is an extraspecial $p$-group.

\end{itemize}

We now list the main results of this article. Our first result describes the $\mathbb{C}$-twist isomorphism classes of groups with generalized corank zero or one.

\begin{lemma}
  \label{lem:s(G)-0-1} For non-abelian groups $G$ of order $p^{n}$ with $s(G) \in \{ 0, 1\}$, every  $\mathbb{C}$-twist isomorphism class is a singleton, i.e. consists of only one group up to isomorphism.
\end{lemma}

In our next result, we describe all non-singleton $\mathbb{C}$-twist isomorphism classes of finite $p$-groups with $s(G)=2$.

\begin{theorem}\label{thm:s(G)=2}
  For non-abelian groups $G$ of order $p^{n}$ with $s(G)=2$, all non-singleton  $\mathbb{C}$-twist isomorphism classes are as follows:

  \begin{itemize} \item [(1)] for any $n \geq 4$, $\mathbb{Q}_{8} \times C_{2}^{(n-3)} \sim_\mathbb{C} E(2) \times C_{2}^{(n-4)}$; 

    \item [(2)] for an odd prime $p$, $E(2) \sim_\mathbb{C} H_1^2 \times C_p \sim_\mathbb{C} \langle a,b~|~a^{p^2}=1,b^{p}=1, [a,b,a]=[a,b,b]=1 \rangle$;

    \item [(3)] for $n \geq 5$, $E(2) \times C_p^{(n-4)} \sim_\mathbb{C} H_1^2 \times C_p^{(n-3)}$;

    \item [(4)] for $n = 2m+1$ and $m \geq 2,$ $H_m^1 \times C_p^{(n-2m-1)} \sim_\mathbb{C} H_m^2 \times C_p^{(n-2m-1)}$;

    \item [(5)] for $n \geq 6$ and $1 < m \leq n/2-1$, $E(2) \times C_p^{(n-2m-2)}  \sim_\mathbb{C}  H_m^1 \times C_p^{(n-2m-1)} \sim_\mathbb{C} H_m^2 \times C_p^{(n-2m-1)}$.
  \end{itemize}

\end{theorem}


See Section~\ref{sec:s(G)leq2} for the proof of Lemma~\ref{lem:s(G)-0-1} and Theorem~\ref{thm:s(G)=2}. The next result describes the non-singleton $\mathbb{C}$-twist isomorphism classes for the groups of order $p^{n}$ with $s(G)=3$. A complete classification of these groups was given by Hatui~\cite[Theorem 1.1]{Hatui}. We refer the reader to Theorem~\ref{sumana} for the details and for the notation appearing in our next result.

\begin{theorem}\label{thm:s(G)=3}   For non-abelian groups $G$ of order $p^{n}$ with $s(G)=3$, all non-singleton $\mathbb{C}$-twist isomorphism classes are as follows:

  \begin{enumerate}
    \item $\phi_3(211)a \sim_{\mathbb C} \phi_3(211)b_1 \sim_{\mathbb C} \phi_3(211)b_{r_p};$
    \item $\phi_2(2111)c \sim_{\mathbb C} \phi_2(2111)d.$
  \end{enumerate}
\end{theorem}

A proof of the above result is included in Section~\ref{sec:s(G) = 3}. 

\begin{remark} In literature, classification of all finite $p$-groups $G$ with corank of $G$ (denoted by $t(G)$) atmost $6$ is also known, see~\cite{MR1140606, MR1705860, Jafari, MR2990893, MR3375732, MR1255666}. By definition, $t(G) \leq 6$ for any non-abelian group $G$ of order $p^n$ implies $n \leq 8$ and $s(G) \leq 5$. Further, $s(G) \in \{4,5\}$ gives $n \leq 4$. Therefore, our above description of $\mathbb C$-twist isomorphism classes along with the known results from literature also gives $\sim_{\mathbb C}$ classes of groups with $t(G) \leq 6$.
\end{remark}

\section{Preliminaries}

We first recall the results of {\bf Clifford theory} regarding the ordinary characters of a finite group. For proofs, see \cite[Chapter~6]{MR2270898}. Clifford theory provides an important connection between the complex characters of a finite group $G$ and its normal subgroups. For a finite group $G$, we use $\irr(G)$ to
denote the set of all inequivalent irreducible representations of $G$. For an abelian group $A$, we also use $\widehat{A}$ to denote $\irr(A)$ and call this to be the set of characters of $A$. For $\chi \in \irr(N)$, where $N$ is a normal subgroup of $G$, we use $\irr(G \mid \chi)$ to denote the set of irreducible representations of $G$ lying above $\chi$, i.e., $\irr(G \mid \chi) = \{ \rho \in \irr(G) \mid \langle \rho|_N, \chi \rangle \neq 0  \}$.

\begin{theorem}
  \label{thm:clifford}  Let $G$ be a finite group and $N$ be a normal subgroup. For any
  irreducible representation $\rho$ of $N$, let $I_G(\rho) =
    \{ g \in G |\,\, \rho^{g} \cong \rho \}$ denote the stabilizer of $\rho$ in $G$. Then
  the following hold:
  \begin{itemize}
    \item[(a)]  The map
      \[ \theta \mapsto \mathrm{Ind}_{I_G(\rho)}^G(\theta)
      \]
      is a bijection of $\irr(I_G(\rho) \mid \rho) $ onto $\irr(G \mid \rho)$.
    \item[(b)] Let $H$ be a subgroup of $G$ containing $N$, and suppose that $\rho$ is
      an irreducible representation of $N$ which has an extension $\tilde{\rho}$ to
      $H$ (i.e. $\tilde{\rho}|_{N} = \rho $). Then the representations $\delta \otimes
        \tilde{\rho}$
      for $\delta \in \irr(H/N)$ are irreducible, distinct for distinct $\delta$ and
      \[
        \mathrm{Ind}^{H}_{N}(\rho) = \oplus_{\delta \in \irr(H/N)} \delta \otimes
        \tilde{\rho}.
      \]
  \end{itemize}
\end{theorem}

Another useful result in this direction is as follows:
\begin{theorem}(\cite{MR2270898}, Corollary 11.22) Let $N \unlhd G$ and suppose $G/N$ is cyclic. Let $\rho \in \irr(N)$ such that $I_G(\rho) = G$. Then $\rho$ is extendible to G.
\end{theorem}

Recall that a finite group $G$ is said to be of \textit{central type} if it has a non-degenerate 2-cocycle $\alpha \in Z^{2}(G,\mathbb{C}^{*})$; or in other words, it has a unique irreducible $\alpha$-projective representation so that the twisted group algebra $\mathbb{C}^{\alpha}G$ is simple. The following result given by Margolis-Schnabel \cite{MS} will be used later to determine $\mathbb{C}$-twist isomorphism classes: 
\begin{lemma} \label{centraltype}
  Let $G$ and $H$ be groups of order $p^4$. Then $G \sim_{\mathbb C} H$ if the following three conditions   are satisfied. 
  \begin{enumerate}
    \item $\mathbb{C}G \cong \mathbb{C}H$
    \item $\Ho^2(G, \mathbb C^\times) \cong \Ho^2(H, \mathbb C^\times)$
    \item $G$ and $H$ are not of central type.
  \end{enumerate}

\end{lemma}

\section{$p$-groups with $s(G) \leq 2$}
\label{sec:s(G)leq2}

\noindent In this section, we study the $\mathbb{C}$-twist isomorphism classes of $p$-groups with $s(G) \leq  2$.
We first deal with the case of $s(G) \in \{0,1\}$.

\begin{proof}[Proof of Lemma~\ref{lem:s(G)-0-1}]
  Niroomand~\cite[Theorem~21, Corollary 23]{N1} proved the following classification of finite non-abelian $p$-groups $G$ with $s(G) \in \{0,1\}$:
  \begin{itemize}
    \item[(a)] A group $G$ has $s(G) = 0$ if, and only if, $G$ is isomorphic to $H_{1}^{1} \times C_{p}^{(n-3)}$.
    \item[(b)] A group $G$ has $s(G) = 1$ if, and only if, $G$ is isomorphic to $D_{8} \times C_{2}^{(n-2)}$ or $C_{p}^{(4)} \rtimes C_{p}\, (p \neq 2)$.
  \end{itemize}

  We remark that in \cite{N1}, Niroomand uses the corank of a group $G$ (denoted $t(G)$)  instead of the generalized corank of $G$. We have used the well known relation $t(G) = s(G) + (n-2)$ for any non-abelian $p$-group $G$ to use the results of \cite{N1}. We obtain Lemma~\ref{lem:s(G)-0-1} by  observing that all of the above groups with fixed $s(G)$ have non-isomorphic complex group algebras.
\end{proof}

The rest of this section is devoted to the $s(G) = 2$ case. The following result from \cite{N4} classifies all the non-abelian groups of order $p^n$ with $s(G) = 2.$

\begin{theorem} (\cite{N4}, Theorem 11) \label{thm:niroo-s(G)=2} Let $G$ be a group of order $p^n$. Then $s(G) = 2$ if, and only if, $G$ is isomorphic to one of the following:

  \begin{itemize}
    \item [(i)] $E(2) \times C_p^{(n-2m-2)}$
    \item [(ii)] $H_1^{2} \times C_p^{(n-3)}$
    \item [(iii)] $Q_8 \times C_2^{(n-3)}$
    \item [(iv)]\begin{itemize}
            \item [(a)] $H_m^1 \times C_p^{(n-2m-1)},$ where $H_m^1$ is an extraspecial $p$-group of order $p^{2m+1}$ and exponent $p~(m \geq 2)$
            \item [(b)] $H_m^2 \times C_p^{(n-2m-1)},$ where $H_m^2$ is an extraspecial $p$-group of order $p^{2m+1}$ and exponent $p^2~(m \geq 2)$ \end{itemize}

    \item [(v)] $\langle a,b~|~a^4=1,b^4=1,[a,b,a]=[a,b,b]=1, [a,b]=a^2b^2 \rangle$ 
    \item [(vi)] $ \langle a,b,c~|~a^2=b^2=c^2=1, abc=bca=cab \rangle$ 
    \item [(vii)] $\langle a,b~|~a^{p^2}=1,b^{p}=1, [a,b,a]=[a,b,b]=1 \rangle$
    \item [(viii)] $C_{p} \times (C_{p}^4 \rtimes_{\theta} C_{p})~(p \neq 2)$
    \item [(ix)] $\langle a,b~|~ a^9=b^3=1, [a,b,a] =1, [a,b,b]=a^6, [a,b,b,b]=1 \rangle $
    \item [(x)] $\langle a,b~|~ a^p=1, b^p=1, [a,b,a]=[a,b,b,a]=[a,b,b,b]=1 \rangle~(p \neq 3).$
  \end{itemize}
\end{theorem}

In order to prove Theorem \ref{thm:s(G)=2}, we need the following lemma:

\begin{lemma}
  \label{lem:C-twist-criterion}
  Suppose $G_1$ and $G_2$ are two groups with isomorphisms $\delta: G_1' \rightarrow G_2',$ $\sigma: G_1/G_1' \rightarrow G_2/G_2'$ and the following short exact sequences for $i \in \{ 1, 2\} :$
  \[
    \xymatrix{
    1 \ar[r] &  \mathrm{Hom}(G_i', \mathbb C^\times) \ar[r] ^{\mathrm{tra}_i} & H^2(G_i/G_i', \mathbb C^\times) \ar[r]\ar[r]^{\mathrm{inf_i}} & H^2(G_i, \mathbb C^\times) \ar[r] &  1.
    }
  \]
  If $\tilde{\delta}:  \Hom(G_2', \mathbb{C}^\times) \rightarrow \Hom(G_1', \mathbb C^\times)$ and $\tilde{\sigma}: \Ho^2(G_1/G_1', \mathbb C^\times) \rightarrow \Ho^2(G_2/G_2', \mathbb C^\times)$ are the induced isomorphisms such that the Figure~\ref{fig:commutative} is commutative, then $G_{1}\sim_{\mathbb{C}}G_{2}$.
  \begin{figure}[h]
    \[\xymatrix{
      1 \ar[r] &  \mathrm{Hom}(G_2', \mathbb C^\times) \ar[r] ^{\mathrm{tra}_2} \ar[d]^{\tilde{\delta}} & H^2(G_2/G_2', \mathbb C^\times) \ar[r] \ar[d]^{\tilde{\sigma}} \ar[r]^{\mathrm{inf_2}} & H^2(G_2, \mathbb C^\times) \ar[r] &  1 \\
      1 \ar[r] &  \mathrm{Hom}(G_1', \mathbb C^\times) \ar[r] ^{\mathrm{tra}_1} & H^2(G_1/G_1', \mathbb C^\times) \ar[r] ^{\mathrm{inf_1}} & H^2(G_1, \mathbb C^\times) \ar[r] &  1
      }
    \]
    \caption{} \label{fig:commutative}
  \end{figure}
\end{lemma}
\begin{proof}
  Our goal is to define an isomorphism $\gamma: \Ho^2(G_2, \mathbb C^\times) \rightarrow \Ho^2(G_1, \mathbb C^\times)$ that gives $\mathbb C$-twist isomorphism between $G_2$ and $G_1$.
  It follows from Theorem $2.9$ in \cite{MR788161} and Theorem $3.1$ in \cite{MR4232687} that the projective representations of $G_i$ are obtained from those of $G_i/G_i'$ via inflation
  and
  \begin{eqnarray}
    \label{eq:cong-group-algebra}
    \mathbb C^\alpha[G_i] \cong \prod_{\mathrm{inf_i}(\beta) = \alpha} \mathbb{C}^\beta[G_i/G_i'].
  \end{eqnarray}

  The map $\tilde{\sigma}$ is an induced isomorphism obtained from $\sigma$. Hence $\mathbb{C}^\beta[G_1/G_1'] \cong \mathbb{C}^{\tilde{\sigma}(\beta)}[G_2/G_2'].$  Therefore, it is sufficient to define an isomorphism $\gamma: \Ho^2(G_2, \mathbb C^\times) \rightarrow \Ho^2(G_1, \mathbb C^\times)$ such that Figure~\ref{fig:commutative} is commutative. Indeed such a $\gamma$ is obtained by defining
  \[
    \gamma([\alpha]) = \mathrm{inf}_1 (\tilde{\sigma}([\alpha_0])) \,\, \mathrm{for}\,\, [\alpha] \in H^2(G_2, \mathbb C^\times),
  \]
  where $[\alpha_0] \in H^2(G_2/G_2', \mathbb C^\times)$ is any element such that $\inf_2([\alpha_0]) = [\alpha]$.
\end{proof}

\begin{proposition}\label{prop1} The distinct $\mathbb{C}$-twist isomorphism classes of groups of order $p^n$ are as follows:
  \begin{itemize}
    \item [(i)] \begin{itemize} \item[(a)] $E(2) \times C_p^{(n-4)} \sim_{\mathbb C} H_1^{2} \times C_p^{(n-3)},$ for $p \neq 2$
            \item[(b)] $E(2) \times C_2^{(n-4)} \sim_{\mathbb C} Q_8 \times C_2^{(n-3)}$  \end{itemize}

    \item [(ii)] \begin{itemize} \item[(a)] For $n = 2m+1$ and $m \geq 2$, $H_m^1 \times C_p^{(n-2m-1)} \sim_{\mathbb C} H_m^2 \times C_p^{(n-2m-1)}$

            \item [(b)] For $n \geq 6$ and $m \geq 2$, $E(2) \times C_p^{(n-2m-2)} \sim_{\mathbb C} H_m^1 \times C_p^{(n-2m-1)} \sim_{\mathbb C} H_m^2 \times C_p^{(n-2m-1)}.$
          \end{itemize}

  \end{itemize}
\end{proposition}
\begin{proof}
  We proceed to prove (i). 
  The proof of (ii) is similar so we only give essential ingredients there and omit the details.

  (i)(a) For simplification of notations, we denote $E(2) \times C_p^{(n-4)}$ and $H_1^2 \times C_p^{(n-3)}$ by $G_1$ and $G_2$, respectively. The groups $G_1$ and $G_2$ have following presentations:
  \[
    G_1 = <x_1,y_1,z_1, \gamma_1, a_1, a_2, \cdots a_{n-4} \mid [x_1,y_1] = z_1 = \gamma_1^p, x_1^p = y_1^p = a_i^p = 1, \gamma_1^{p^2} = 1>
  \]
  \[
    G_2 = < x_2,y_2,z_2, b_1, \cdots, b_{n-3} \mid [x_2,y_2] = x_2^p = z_2, x_2^{p^2} = y_2^p = z_2^p =  b_i^p = 1>.
  \]
  Therefore, $G_i'  \cong C_{p}$  and $G_i/G_i' \cong  C_{p}^{ n-1}$ for $i \in \{1, 2\}$. Also, in view of Proposition 1.3 of \cite{MR788161},
  \[
    \Ho^2(G_i/G_i', \mathbb C^\times) \cong C_{p}^{ \frac{(n-1)(n-2)}{2}}, \,\,  \Ho^2(G_i, \mathbb C^\times) \cong C_{p}^{ \frac{(n-1)(n-2)}{2}-1}.
  \]
  This yields the following short exact sequences for $i \in \{1, 2\}$:
  \[
    \xymatrix{
    1 \ar[r] &  \mathrm{Hom}(G_i', \mathbb C^\times) \ar[r] ^{\mathrm{tra}_i} & H^2(G_i/G_i', \mathbb C^\times) \ar[r]\ar[r]^{\mathrm{inf_i}} & H^2(G_i, \mathbb C^\times) \ar[r] &  1.
    }
  \]
  We now define $\delta: G_1' \rightarrow G_2'$ and $\sigma: G_1/G_1' \rightarrow G_2/G_2'$ such that the Figure~\ref{fig:commutative} is commutative. Define $\delta$ by $\delta(z_1) = z_2$ and $\sigma$ by
  \[
    \sigma(x_1 G_1^{'}) = x_2 G_2^{'},   \sigma(y_1 G_1^{'}) = y_2 G_2^{'},   \sigma(\gamma_1 G_1^{'}) = b_{n-3} G_2^{'},  \sigma (a_i G_1^{'}) = b_i G_2^{'},
  \]
  for all $i \in \{1,\ldots, n-4\}$. We now describe transgression maps for these groups.

  Define a section $s_1:G_1/G_1' \rightarrow G_1$ by
  \[
    s_1(x_1^i y_1^j \gamma_1^k a_1^{r_1} \cdots a_{n-4}^{r_{n-4}}G_1' ) =  x_1^i y_1^j \gamma_1^k a_1^{r_1} \cdots a_{n-4}^{r_{n-4}}
  \]
  For $u=  x_1^i y_1^j \gamma_1^k a_1^{r_1} \cdots a_{n-4}^{r_{n-4}} G_1^{'}$ and $v = x_1^{i'} y_1^{j'} \gamma_1^{k'} a_1^{r'_1} \cdots a_{n-4}^{r'_{n-4}} G_1^{'},$ we have
  \[
    s_1(u)s_1(v)s_1(uv)^{-1} = \gamma_1^{-p j i'}.
  \]
  Hence a representative of $[\mathrm{tra_1}(\chi)]$ is given by $\tra_1(\chi)(u, v)= \chi(z_1^{-ji'})$ for $\chi \in \Hom(G_1', \mathbb C^\times)$.
  Define a section $s_2:G_2/G_2' \rightarrow G_2$ by
  \[
    s_2(x_2^i y_2^j b_1^{r_1} \cdots b_{n-3}^{r_{n-3}}G_2' ) =  x_2^i y_2^j b_1^{r_1} \cdots b_{n-3}^{r_{n-3}}
  \]
  For $u=  x_2^i y_2^j b_1^{r_1} \cdots b_{n-3}^{r_{n-3}} G_2^{'}$ and $v = x_2^{i'} y_2^{j'} b_1^{r'_1} \cdots b_{n-3}^{r'_{n-3}} G_2^{'},$ we have
  \[
    s_2(u)s_2(v)s_2(uv)^{-1} = z_2^{-ji'}.
  \]
  Therefore, a representative of $[\mathrm{tra_2}(\chi)]$ is given by $\mathrm{tra_2}(\chi) (u,v) = \chi(z_2^{-ji'})$ for $\chi \in \Hom(G_2', \mathbb C^\times)$. This combined with the given isomorphisms $\delta$ and  $\sigma$ gives the commutativity of Figure~\ref{fig:commutative}. Now, the result follows as a direct consequence of Lemma~\ref{lem:C-twist-criterion}.

  For $(i)(b)$, proof is along the same lines as that of $(i)(a)$ with only difference that $Q_8 \times C_2^{(n-3)}$ has the following presentation:  
  \[
   \langle a, b, c, b_1, \cdots, b_{n-3} \ | \ a^4=1, a^2=b^2=c, b^{-1}ab = ca, b_i^{2} = 1 \ \forall \ 1 \leq i \leq (n-3) \rangle.
  \]
  We leave the rest of the details for the reader.

  (ii) We denote the groups $E(2) \times C_p^{(n-2m-2)}$, $H_m^1 \times C_p^{(n-2m-1)}$ and $H_m^2 \times C_p^{(n-2m-1)}$ by $G^m_{1}$, $G^m_{2}$ and $G^m_{3}$, respectively. Note that if $m \neq m'$, then for any $i,j \in \{ 1, 2,3\},$ the complex group algebras of $G_i^m$ and $G_j^{m'}$ are not isomorphic. Further, observe that for any $m \geq 2$, the order of $G_{1}^m$ is $p^n$ such that $n \geq 6.$ Therefore, for some $n \geq 6,$ if $G_i^{m}$ is $\mathbb{C}$-twist isomorphic to $G_j^{m'},$ then it implies that $m = m'$. Similarly, when $n = 5,$ a necessary condition for the $\mathbb{C}$-twist isomorphism of $G_1^{m}$ and $G_2^{m'}$ is that $ m = m'$. Therefore, now onwards we fix $m$ and prove the result.

  The commutator subgroup of $G_i^m$ is central and is isomorphic to $C_{p};$ and $G_i^m/(G_i^m)' \cong C_p^{(n-1)}.$ Further, since for any $i \in \{1, 2, 3\},$ $\Ho^2(G_i^m, \mathbb C^\times) \cong  C_p^{ \frac{n^2-3n}{2} }$, we get the following short exact sequences:
  \[
    \xymatrix{
    1 \ar[r] &  \mathrm{Hom}(G_i', \mathbb C^\times) \ar[r] ^{\mathrm{tra}_i} & H^2(G_i/G_i', \mathbb C^\times) \ar[r]\ar[r]^{\mathrm{inf_i}} & H^2(G_i, \mathbb C^\times) \ar[r] &  1.
    }
  \]
  As in (i), the proof of $\mathbb C$-twist isomorphism follows by considering the image of the $\mathrm{tra_i}$ for $i \in \{1, 2, 3\}$ and by proving the commutativity of Figure~\ref{fig:commutative}.
  This is obtained by using the following presentation of groups $G_i^m.$
  \begin{eqnarray*}
    G_1^m & = &  <x_1, \cdots, x_m,y_1, \cdots y_m,z, \gamma, a_1, a_2, \cdots a_{n-2m-2} \mid [x_i,y_i] = z = \gamma^p, \\ & &  x_i^p = y_i^p = a_i^p = 1, \gamma^{p^2} = 1>. \\
    G_2^m &  = & <x_1, \cdots, x_m,y_1, \cdots y_m,z, a_1, a_2, \cdots, a_{n-2m-1} \mid [x_i,y_i] = z, \\ & & x_i^p = y_i^p = a_i^p = 1>. \\
    G_3^m & =  & <x_1, \cdots, x_m,y_1, \cdots y_m,z, a_1, a_2, \cdots, a_{n-2m-1} \mid [x_i,y_i] = z = x_m^p = y_m^p, \\ & & x_i^p = y_i^p = z^p = 1 (1 \leq i \leq m-1), a_i^p = 1>.
  \end{eqnarray*}

  Below we calculate $\mathrm{tra}_1$ explicitly and leave the details for $\tra_2$ and $\tra_3$ as those are similar. By the given presentation of $G_1^m$, we have $(G_1^m)' = <\gamma^p>$. Define a section $s:G_1^m/(G_1^m)' \rightarrow G_1^m$ by
  \[
    s(x_1^{i_1} \cdots x_m^{i_m} y_1^{j_1} \cdots y_m^{j_m} \gamma^k a_1^{r_1} \cdots a_{n-2m-2}^{r_{n-2m-2}} (G_1^m)' ) =  x_1^{i_1} \cdots x_m^{i_m} y_1^{j_1} \cdots y_m^{j_m} \gamma^k a_1^{r_1} \cdots a_{n-2m-2}^{r_{n-2m-2}}
  \]
  Note that for any two elements 
  $$u=  x_1^{i_1} \cdots x_m^{i_m} y_1^{j_1} \cdots y_m^{j_m} \gamma^k a_1^{r_1} \cdots a_{n-2m-2}^{r_{n-2m-2}}(G_1^m)^{'}, v = x_1^{i'_1} \cdots x_m^{i'_m} y_1^{j'_1} \cdots y_m^{j'_m} \gamma^{k'} a_1^{r'_1} \cdots a_{n-2m-2}^{r'_{n-2m-2}}(G_1^m)^{'}$$ of $G_1^m/(G_1^m)^{'},$ we have $   s(u)s(v)s(uv)^{-1} = \gamma^{-p \sum_{l=1}^mj_l i'_l}.$ Therefore, for any $\chi \in \mathrm{Hom}((G_1^m)', \mathbb C^\times)$, a representative of $[\mathrm{tra_1}(\chi)]$ is given by $\mathrm{tra_1}(\chi)(u,v) = \chi(z^{-\sum_{l=1}^mj_l i'_l}).$ By a similar computation of $\tra_2$ and $\tra_3$, we obtain that  for all $i \in \{1, 2, 3\},$ the groups $G_i^m$ pairwise satisfy the hypothesis of Lemma~\ref{lem:C-twist-criterion} and hence are $\mathbb C$-twist isomorphic.
\end{proof}

We now complete the details regarding the $\mathbb{C}$-twist isomorphism classes for $p$-groups with $s(G)=2$.

\begin{proof}[\bf{Proof of Theorem~\ref{thm:s(G)=2}}]

  It follows from Theorem \ref{thm:niroo-s(G)=2} that for any fixed $p$, there exists only one $p$-group of order $p^3$ with $s(G) = 2$ and so it forms a singleton $\mathbb C$-twist class. We now consider cases when $n \geq 4.$
\vspace{.2cm}\\ 
$n=4$: Any group of order $p^4$ with $s(G) = 2$ is isomorphic to one of the following:
  \begin{itemize}
    \item $E(2)$
    \item $H_1^{2} \times C_p, p \neq 2$
    \item $Q_8 \times C_2$
    \item  $\langle a,b~|~a^4=1,b^4=1,[a,b,a]=[a,b,b]=1, [a,b]=a^2b^2 \rangle$ 
    \item $ \langle a,b,c~|~a^2=b^2=c^2=1, abc=bca=cab \rangle$ 
    \item $\langle a,b~|~a^{p^2}=1,b^{p}=1, [a,b,a]=[a,b,b]=1 \rangle$
    \item  $\langle a,b~|~ a^9=b^3=1, [a,b,a] =1, [a,b,b]=a^6, [a,b,b,b]=1 \rangle $
    \item  $\langle a,b~|~ a^p=1, b^p=1, [a,b,a]=[a,b,b,a]=[a,b,b,b]=1 \rangle $ $(p \neq 3).$
  \end{itemize}

  Here the group $\langle a,b,c~|~a^{2}=b^{2}=c^{2} = 1,abc=bca=cab \rangle$ is isomorphic to $E(2).$ As mentioned in Theorem $4.3$ in \cite{MS}, any non-singleton $\mathbb{C}$-twist isomorphism class of groups of order $p^4$ consists of two groups, when $p=2;$ and of three groups, when $p$ is an odd prime.
  Thus comparing with the groups given in Table 3 and Table 4 in \cite{MS}, we obtain the following non-singleton $\mathbb C$-twist isomorphism classes of groups of order $p^4$ with generalized corank $2$:

  \begin{itemize} \item  $\mathbb{Q}_{8} \times C_{2} \sim_\mathbb{C} E(2),$ when $p=2$

    \item $E(2) \sim_\mathbb{C} H_1^1 \times C_p \sim_\mathbb{C} \langle a,b~|~a^{p^2}=1,b^{p}=1, [a,b,a]=[a,b,b]=1 \rangle$, when $p$ is odd.
  \end{itemize}
  Thus each of the remaining groups of order $p^4$ in the above list constitutes a $\mathbb C$-twist isomorphism class of size $1.$
  \vspace{.2cm}\\
  \textbf{$n \geq 5:$}
  Any group of order $p^n$ with $n \geq 5$ and $s(G) = 2$ is isomorphic to one of the following:

  \begin{itemize}
    \item  $E(2) \times C_p^{(n-2m-2)}$
    \item  $H_1^{2} \times C_p^{(n-3)}, p \neq 2$
    \item  $Q_8 \times C_2^{(n-3)}$
    \item $H_m^1 \times C_p^{(n-2m-1)}$
    \item $H_m^2 \times C_p^{(n-2m-1)}$
    \item $C_{p} \times (C_{p}^4 \rtimes_{\theta} C_{p}),~ p \neq 2.$
  \end{itemize}

  The derived subgroup of $C_{p} \times (C_{p}^4 \rtimes_{\theta} C_{p}),~p \neq 2$ is of order $p^2$; whereas the derived subgroup of the rest of the groups in the above list is of order $p$. Therefore, by comparing the complex group algebras,
  we obtain that for a fixed odd prime $p,$ the group $C_{p} \times (C_{p}^4 \rtimes_{\theta} C_{p})$ forms a singleton $\mathbb C$-twist isomorphism class. Finally, Proposition~\ref{prop1} completes the classification of the rest of the groups into $\mathbb C$-twist isomorphism classes.
\end{proof}

\section{$p$-groups with $s(G)=3$}
\label{sec:s(G) = 3}

\noindent In this section, we proceed with the determination of the $\mathbb{C}$-twist isomorphism classes of the $p$-groups with $s(G) = 3.$ The following result, using the notations of \cite{james}, gives a complete list of $p$-groups with $s(G) = 3$.

\begin{theorem} (\cite{Hatui}, Theorem 1.1) \label{sumana}
  Let $G$ be a finite non-abelian $p$-group of order $p^{n}$ with $s(G)=3$. Let $r_p$ be the smallest positive integer which is a non-quadratic residue mod $(p)$. 
  \begin{itemize}
    \item[(a)] For an odd prime $p$, $G$ is isomorphic to one of the following groups:

      \begin{itemize}
        \item [(i)] $\phi_{2}(22)= \langle \alpha,~\alpha_{1},~\alpha_{2}~|~[\alpha_{1},\alpha]=\alpha^{p}=\alpha_{2},~\alpha_{1}^{p^{2}}=\alpha_{2}^{p}=1 \rangle$

        \item [(ii)] $\phi_{3}(211)a= \langle \alpha, ~\alpha_{1},\alpha_{2},~\alpha_{3}~|~[\alpha_{1},\alpha]=\alpha_{2},~[\alpha_{2},\alpha] =\alpha^{p}=\alpha_{3}, \alpha_{1}^{(p)}=\alpha_{2}^{p}=\alpha_{3}^{p}=1\rangle $

        \item [(iii)] $\phi_{3}(211)b_{r} = \langle \alpha, ~\alpha_{1},\alpha_{2},\alpha_{3}~|~[\alpha_{1},\alpha]=\alpha_{2},~[\alpha_{2},\alpha]^r =\alpha_{1}^{(p)}=\alpha_{3}^r,~\alpha^{p}=\alpha_{2}^{p}=\alpha_{3}^{p}=1 \rangle,$ 
        where $r$ is either $1$ or $r_p.$

        \item [(iv)] $\phi_{2}(2111)c= \phi_{2}(211)c \times C_{p}$, where $\phi_{2}(211)c = \langle \alpha,\alpha_{1},\alpha_{2}~|~[\alpha_{1},\alpha]=\alpha_{2},~\alpha^{p^{2}}=\alpha_{1}^{p}=\alpha_{2}^{p}=1 \rangle$

        \item [(v)] $\phi_{2}(2111)d= ES_{p}(p^{3}) \times C_{p^{2}}$

        \item [(vi)] $\phi_{3}(1^{5}) = \phi_{3}(1^{4}) \times C_{p}$, where $\phi_{3}(1^{4})= \langle \alpha,  \alpha_{1},\alpha_{2},\alpha_{3}~|~[\alpha_{i},\alpha]=\alpha_{i+1},~\alpha^{p}=\alpha_{i}^{(p)}=\alpha_{3}^{p}=1~(i=1,2) \rangle$

        \item [(vii)] $\phi_{7}(1^{5})= \langle ~\alpha, ~\alpha_{1},\alpha_{2},\alpha_{3}, \beta ~|~ ~[\alpha_{i},\alpha]=\alpha_{i+1}, [\alpha_{1},\beta] =\alpha_{3},~\alpha^{p}=\alpha_{1}^{(p)}=\alpha_{i+1}^{p}=\beta^{p}=1~(i=1,2) \rangle$

        \item [(viii)] $\phi_{11}(1^{6}) =  \langle ~\alpha_{1}, ~\beta_{1},\alpha_{2},\beta_{2}, \alpha_{3}, \beta_{3}~|~~[\alpha_{1},\alpha_{2}]=\beta_{3},~[\alpha_{2},\alpha_{3}] =\beta_{1},~[\alpha_{3},\alpha_{1}] =\beta_{2},\alpha_{i}^{(p)}=\beta_{i}^{p}~(i=1,2,3) \rangle$

        \item [(ix)] $\phi_{12}(1^{6})=ES_{p}(p^{3}) \times ES_{p}(p^{3})$

        \item [(x)] $\phi_{13}(1^{6})  = \langle \alpha_{1},\alpha_{2},\alpha_{3}, \alpha_{4}, \beta_{1},\beta_{2} ~|~[\alpha_{1},\alpha_{i+1}]=\beta_{i},~[\alpha_{2},\alpha_{4}]= \beta_{2},~\alpha_{i}^{p}=\alpha_{3}^{p}=\alpha_{4}^{p}=\beta_{i}^{p}=1 (i=1,2)  \rangle$

        \item [(xi)]  $\phi_{15}(1^{6})= \langle \alpha_{1},\alpha_{2},\alpha_{3}, \alpha_{4}, \beta_{1},\beta_{2} ~|~[\alpha_{1},\alpha_{i+1}]=\beta_{i},~[\alpha_{3},\alpha_{4}]= \beta_{1},~[\alpha_{2},\alpha_{4}]= \beta_{2}^{g},~\alpha_{i}^{p}=\alpha_{3}^{p}=\alpha_{4}^{p}=\beta_{i}^{p}=1 (i=1,2)  \rangle$, where $g$ is the smallest positive integer which is a primitive root modulo $p$

        \item [(xii)] $(C_{p}^{(4)} \rtimes C_{p}) \times C_{p}^{2}$.
      \end{itemize}
    \item[]
    \item[(b)] For $p=2$, $G$ is isomorphic to one of the following groups:
    \item[]

      \begin{itemize}
        \item [(xiii)] $C_{2}^{4} \rtimes C_{2}$

        \item [(xiv)] $C_{2} \times ((C_{4}\times C_{2}) \rtimes C_{2})$

        \item [(xv)] $C_{4} \rtimes C_{4}$

        \item [(xvi)] $D_{16}$, the dihedral group of order 16.   \end{itemize}

  \end{itemize}
\end{theorem}

As mentioned earlier, the $\mathbb{C}$-twist isomorphism classes of the groups of order $p^4$ were described by Margolis-Schnabel~\cite{MS}. We now consider the groups of order $p^5$ with $s(G) = 3$.  Let $H_1$ and $H_2$ denote the groups $\phi_2(2111)d $ and $\phi_2(2111)c,$ respectively. We proceed to prove that $H_1$ and $H_2$ are $\mathbb{C}$-twist isomorphic.  We use the following general result to prove this. We refer the reader to 
\cite[Chapter~3]{MR788161} for results regarding the existence and construction of the representation group of a finite group.

\begin{lemma}\label{l1}
  Let $G_1$ and $G_2$ be two finite groups with $\widetilde{G_1}$ and $\widetilde{G_2}$ as their representation groups respectively. Let $A_1, A_2$ be central subgroups of $\widetilde{G_1}$ and $\widetilde{G_2}$ respectively such that $\widetilde{G_i}/A_i \cong G_i$ and the transgression maps $\mathrm{tra_i}: \mathrm{Hom}(A_i , \mathbb C^\times) \rightarrow \Ho^2(G_i, \mathbb C^\times)$ are isomorphisms. Let $\sigma: \mathrm{Hom}(A_1, \mathbb C^\times) \rightarrow \mathrm{Hom}(A_2, \mathbb C^\times)$ be an isomorphism such that the following sets are in a dimension preserving bijection for every $\chi \in \mathrm{Hom}(A_1, \mathbb C^\times)$:
  $$
    \mathrm{Irr}(\widetilde{G_1} \mid \chi) \leftrightarrow \mathrm{Irr}(\widetilde{G_2} \mid \sigma(\chi)).
  $$
  Then $G_1$ and $G_2$ are $\mathbb C$-twist isomorphic.
\end{lemma}

\begin{proof}
For $\chi \in \Hom (A_1, \mathbb C^\times)$, we denote $\tra_1(\chi)\in \Ho^2(G_1, \mathbb C^\times)$  by $\alpha_\chi$. We note that $$
    \mathrm{Irr}(\widetilde{G_1} \mid \chi) \leftrightarrow \mathrm{Irr}(\widetilde{G_2} \mid \sigma(\chi)).
  $$  gives
  $$
    \mathbb C^{\alpha_\chi}[ G_1]  \cong \mathbb C^{\alpha_{\sigma(\chi)}}[ G_2 ].
  $$
  Therefore, the map $\alpha_\chi \mapsto \alpha_{\sigma(\chi)}$ between $\Ho^2(G_1, \mathbb C^\times)$ and $\Ho^2(G_2, \mathbb C^\times)$ gives the required $\mathbb C$-twist isomorphism between $G_1$ and $G_2$.
\end{proof}

Note that $H_1 = E_1 \times C_{p^{2}}$ and $H_2 = \langle \alpha, \alpha_1, \alpha_2 \mid [\alpha_1, \alpha ] = \alpha_2, \alpha^{p^2} = \alpha_1^p = \alpha_2^p = 1 \rangle \times \langle \alpha_3 \rangle$. Define the following groups:
\begin{eqnarray*}
  \widetilde{H_1} &  =  &  \langle  \alpha_1, \alpha_2, \alpha_3, \alpha_4, x,y,z,\alpha \mid [x,y] = z, [x,z] = \alpha_1, [y,z] = \alpha_2, \\ & &    [x,\alpha] = \alpha_3,  [y,\alpha] = \alpha_4, x^p = y^p = z^p = \alpha_{i}^{p}=\alpha^{p^2} = 1 \rangle.
\end{eqnarray*}
and
\begin{eqnarray*}
  \widetilde{H_2} & = &   \langle x,y,z,w, \alpha, \alpha_1, \alpha_2,\alpha_{3} \mid [\alpha_1, \alpha] = \alpha_2, [\alpha_1, \alpha_2] = x, [\alpha, \alpha_2]   = y, \\ & &  [\alpha_3, \alpha_{1}] = z, [\alpha_3,\alpha] = w,
  \alpha^{p^{2}}=\alpha_i^p = x^{p}=y^{p}=z^{p}=w^{p} = 1 \rangle.
\end{eqnarray*}

The following lemma plays a very crucial role towards this end:
\begin{lemma} \label{l2}
  The groups $\widetilde{H_1}$ and $\widetilde{H_2}$ are representations groups of $H_1$ and $H_2$, respectively.
\end{lemma}

\begin{proof} Here, we give a proof for $H_1$. The proof for $H_2$ is along the same lines so we omit that part. The group $H_1$ has the following presentation:
  \[
    H_1 = \langle x,y,z, \alpha \mid [x,y] = z, x^p = y^p = \alpha^{p^2} = 1 \rangle.
  \]
  Consider the projection map from $\widetilde{H_1} $ onto $H_1$  obtained by mapping $x,y,z,\alpha $ to $x,y,z,\alpha$ respectively and all $\alpha_i$ to $1$. Let $K_1$ be the kernel of this projection map. Then $|K_1| = |\Ho^2(H_1, \mathbb C^\times)| = p^4$ and $K_1 \subseteq Z(\widetilde{H_1}) \cap [\widetilde{H_1}, \widetilde{H_1}]$. Therefore, by \cite[Theorem~3.7 (Chapter 3)]{MR788161}, $\widetilde{H_1}$ is a representation group of $H_1$.
\end{proof}

\begin{proposition}\label{p1} The groups $\widetilde{H_1}$ and $\widetilde{H_2}$ satisfy the following:
  \[
    \mathbb{C}[\widetilde{H_1}] \cong  \mathbb{C}[\widetilde{H_2}] \cong \mathbb C^{\oplus p^4 } \oplus (\mathrm{M}_p(\mathbb C))^{\oplus (p^3+4p^2-p+1)p^2(p-1)} \oplus (\mathrm{M}_{p^2}(\mathbb C))^{\oplus p^3(p-1)^3(p+2)}.
  \]
  Furthermore, for the subgroups $A = \langle \alpha_1, \alpha_2, \alpha_3, \alpha_4 \rangle$ and $B = \langle x,y,z,w \rangle$ of $\widetilde{H_1}$ and $\widetilde{H_2}$ respectively, there exists an isomorphism $\sigma: \widehat{A} \rightarrow \widehat{B}$ such that the following sets are in a dimension preserving bijection for every $\chi \in \hat{A}$:
  $$
    \mathrm{Irr}(\widetilde{H_1} \mid \chi) \leftrightarrow \mathrm{Irr}(\widetilde{H_2} \mid \sigma(\chi)).
  $$

\end{proposition}

\begin{proof}
  {\bf Representations of $\widetilde{H_1}$:} We first justify the representations of $\widetilde{H_1}$. By the definition of $\widetilde{H_1}$, the derived subgroup of $\widetilde{H_{1}}$ (denoted $\widetilde{H_{1}}'$) is $\langle \alpha_{1},\alpha_{2},\alpha_{3},\alpha_{4}, z\rangle$ and the center of $\widetilde{H_{1}}$ is $\langle \alpha_{1},\alpha_{2},\alpha_{3},\alpha_{4}, \alpha^p \rangle$. By considering the quotient group $\widetilde{H_{1}}$/$\widetilde{H_{1}}'$, we obtain that $\widetilde{H_1}$ has exactly $p^{4}$ number of one-dimensional representations.

  We next consider the abelian normal subgroup $N = \langle \alpha_{1},\alpha_{2},\alpha_{3},\alpha_{4},\alpha,z \rangle$ of $\widetilde{H_{1}}$. The group $N$ has order $p^7$. Hence every irreducible representation of $\widetilde{H_1}$ has dimension either $1$, $p$ or $p^2$. We have already justified all one-dimensional representations. Our next goal is to determine all $p$ and $p^2$ dimensional representations.

  Let $\chi \in \operatorname{Irr}(N)$ such that $\chi(\alpha_{1})= \xi^{i_{1}}$, $\chi(\alpha_{2})= \xi^{i_{2}}$, $\chi(\alpha_{3})= \xi^{i_{3}}$ and $\chi(\alpha_{4})= \xi^{i_{4}}$, where $\xi$ is a primitive $p^{th}$ root of unity and $0 \leq i_{1},i_{2},i_{3},i_{4} \leq (p-1)$. We determine the stabilizer of $\chi$ in $\widetilde{H_1}$, denoted by $I_{\wH1}(\chi).$ Recall
  $I_{\wH1}(\chi) = \{g \in \wH1 \mid \chi^g = \chi\}.$
  By definition of $\chi$, $N \leq I_{\wH1}(\chi)$. Next, consider $g = x^{i}y^{j}n$, with $n \in N$.
  Recall that every $m \in N$ satisfes $m = \alpha^{e}z^{f}h$ for some $h \in Z(\widetilde{H_1}).$ Therefore, we have $$\chi^{x^{i}y^{j}}(\alpha^{e}z^{f}h) = \chi(\alpha^{e}z^{f}h)$$ if and only if $\chi(x^{i}y^{j}\alpha^{e}z^{f} y^{-j}x^{-i}) =
    \chi(\alpha^{e}z^{f}).$ Since

  \begin{eqnarray*}
    x^{i}(y^{j}\alpha^{e})z^{f} y^{-j}x^{-i} & = &
    x^{i}(\alpha^e y^{j}\alpha_4^{je})z^{f}y^{-j}x^{-i} =
    \alpha^e x^{i} \alpha_3^{ie}y^{j}\alpha_4^{je}z^{f}y^{-j}x^{-i} \\ & = &
    \alpha^e x^{i} y^{j}z^{f}y^{-j}x^{-i}\alpha_3^{ie}\alpha_4^{je} = \alpha^e x^{i} z^f y^{j} \alpha_2^{jf}y^{-j}x^{-i}\alpha_3^{ie}\alpha_4^{je} \\ & = & \alpha^e z^{f} \alpha_1^{if} \alpha_2^{jf} \alpha_3^{ie}\alpha_4^{je},
  \end{eqnarray*}

  we obtain that $x^i y^j n \in I_{\wH1}(\chi)$ if and only if $\chi(\alpha_1^{if} \alpha_2^{jf} \alpha_3^{ie}\alpha_4^{je}) = 1$. This is equivalent to saying that $\xi^{i_1if+i_2jf+i_3ie+i_4je} = 1.$

  We now consider various cases of irreducible characters of $N$. Note that in each of the cases discussed below, all the computations for $i$ and $j$ are done modulo $p$.

  \begin{itemize}
    \item \textbf{Case I:} Consider the characters $\chi$ of $N$ such that $i_1 = i_2 = i_3 = i_4 = 0$. These are exactly $p^3$ in number. Among these there are $p^2$ characters which act trivially on $z$ and hence on $\widetilde{H_{1}}'$. These give $p^4$ number of one dimensional characters of $\wH1$. The other $(p^3-p^2)$ character of $N$ of this case give irreducible character of $\wH1$ of dimension $p$ and therefore we obtain $p^4-p^3$ many characters of $\wH1$ of dimension $p$.
    \item \textbf{Case II:} When any three of $i_1, i_2, i_3, i_4$ are $0$ and the fourth one is non-zero, then $g = x^i y^j n \in I_{\wH1}(\chi)$ if and only if either $i$ or $j$ is $0$. Thus $|I_{\wH1}(\chi)| = p^8.$

    \item \textbf{Case III:} When any two of $i_1, i_2, i_3, i_4$ are $0$ and the other two are non-zero, then $i$ is a non-zero multiple of $j$ or one of them is zero and other can take any value. Hence $|I_{\wH1}(\chi)| = p^8.$

    \item \textbf{Case IV:} When any three of $i_1, i_2, i_3, i_4$ are non-zero and the fourth one is $0$, then $i = j = 0$ and hence $I_{\wH1}(\chi) = N.$

    \item \textbf{Case V:} Assume that each $i_t,$ where $t \in \{1, 2, 3, 4\},$ is non-zero. When $f = 0$ and $e = 1, i = \frac{-i_4j}{i_3};$ and when $e = 0$ and $f = 1, i = \frac{-i_2j}{i_1}.$ Therefore, $\frac{-i_4j}{i_3} =  \frac{-i_2j}{i_1},$ which holds if and only if $(i_4i_1-i_2i_3)j = 0.$ Now if $i_4i_1 \neq i_2i_3,$ then $i = j =0$ and hence $I_{\wH1}(\chi) = N.$
          Here note that $(p-1)^3$ many characters of $N$ satisfy $i_4i_1 = i_2i_3$ and their inertia 
          group is of order $p^8.$ Thus the remaining $((p-1)^4-(p-1)^3)$ characters have inertia group $N.$

  \end{itemize}

  Considering the case of $(p^7-p^3)$ many characters of $N,$ discussed in Cases II-V,
  the inertia group of 
  $ p^3(p-1)^3(p+2)$ characters is $N,$ and for the other 
$ p^3(p-1)(p^2+4p-1)$ characters, it is of order $p^8.$ By using Clifford theory, we obtain
  \[
    \mathbb C[\wH1] \cong \mathbb C^{\oplus p^4 } \oplus (\mathrm{M}_p(\mathbb C))^{\oplus (p^3+4p^2-p+1)p^2(p-1)} \oplus (\mathrm{M}_{p^2}(\mathbb C))^{\oplus p(p-1)^3(p+2)}.
  \]

  \noindent {\bf Representations of $\widetilde{H_{2}}$:} We have $\widetilde{H_{2}}' = \langle x, y, z, w, \alpha_{2} \rangle$ and $Z(\widetilde{H_{2}}) = \langle \alpha^p, x, y, z, w \rangle$. Clearly, there are exactly $p^{4}$ number of linear characters of $\widetilde{H_{2}}$. Consider the subgroup $$M = \langle \alpha_{2},\alpha_{3}, x, y, z, w, \alpha^p \rangle$$
  of $\wH2$. This is an abelian normal group of order $p^7$. Let $\chi \in \irr(M)$ such that $\chi(x)= \xi^{i_{1}}$, $\chi(y)= \xi^{i_{2}}$, $\chi(z)= \xi^{i_{3}}$ and $\chi(w)= \xi^{i_{4}},$ where $\xi$ is a primitive $p^{th}$ root of unity and $0 \leq i_{1},i_{2},i_{3},i_{4} \leq (p-1)$. Let $g = \alpha^{i} \alpha_1^{j}m$, where $m \in M,$ be an element of $I_{G}(\chi).$
  Therefore, for any $m^{'} = \alpha_2^{k}\alpha_3^{l}h \in M,$ where $h \in Z(\widetilde{H_1}),$ we have $\chi^{\alpha^{i} \alpha_1^{j}}(\alpha_2^{k}\alpha_3^{l}h) = \chi(\alpha_2^{k}\alpha_3^{l}h).$ Now similar computations yield that we must have $\xi^{i_1jk+i_2ik+i_3(-jl)+i_4(-il)} = 1$ and hence the following cases arise:
  \begin{itemize}
    \item \textbf{Case I:} When any three of $i_1, i_2, i_3, i_4$ are $0$ and the fourth one is non-zero, then either $i$ or $j$ is $0.$ Thus $|I_G(\chi)| = p^8.$

    \item \textbf{Case II:} When any two of $i_1, i_2, i_3, i_4$ are $0$ and the other two are non-zero, then $i$ is a non-zero multiple of $j$ or one of them is zero and other can take any value. Hence $|I_G(\chi)| = p^8.$

    \item \textbf{Case III:} When any three of $i_1, i_2, i_3, i_4$ are non-zero and the fourth one is $0$, then $i = j = 0$ and hence $I_G(\chi) = M.$

    \item \textbf{Case IV:} Assume that each $i_t,$ where $t \in \{1, 2, 3, 4\},$ is non-zero. When $l = 0$ and $k = 1, i = \frac{-i_1j}{i_2};$ and when $k = 0$ and $l = 1, i = \frac{-i_3j}{i_4}.$ Therefore, $\frac{-i_3j}{i_4} =  \frac{-i_1j}{i_2},$ which holds if and only if $(i_4i_1-i_2i_3)j = 0.$ Now if $i_4i_1 \neq i_2i_3,$ then $i = j =0$ and hence $I_G(\chi) = M.$ On the other hand, if $i_4 = \frac{i_2i_3}{i_1},$ then $|I_G(\chi)| = p^8.$
          Therefore, in this case the inertia group of $(p-1)^3$ many characters is of order $p^8$ and for the remaining $((p-1)^4-(p-1)^3)$ characters it is $M.$
  \end{itemize}

  Considering all the above cases along with the Clifford theory gives
  \[
    \mathbb C[\wH2] \cong \mathbb C^{\oplus p^4 } \oplus (\mathrm{M}_p(\mathbb C))^{\oplus (p^3+4p^2-p+1)p^2(p-1)} \oplus (\mathrm{M}_{p^2}(\mathbb C))^{\oplus p(p-1)^3(p+2)}.
  \]
  The required isomorphism is also obtained from the above construction.
\end{proof}


\begin{proposition} \label{finalprop}
  The groups $H_1$ and $H_2$
  are $\mathbb C$-twist isomorphic.
\end{proposition}
\begin{proof}
	This follows from Lemma \ref{l1}, Lemma \ref{l2} and Proposition \ref{p1}. 
\end{proof}

\begin{lemma} \label{lastlem}
	$\mathbb C[\phi_{3}(1^{5})] \ncong \mathbb C[\phi_{7}(1^{5})]$. 
\end{lemma}
\begin{proof}
It follows from the presentations of $\phi_{3}(1^{5})$ and $\phi_{7}(1^{5})$ that the nilpotency class of both the groups is $3.$ Now, consider the abelian normal subgroup $N_1 = \langle \alpha_1, \alpha_2, \alpha_3, \beta \rangle$ of $\phi_3(1^5).$ Since it is of index $p,$ each irreducible representation of $\phi_3(1^5)$ is of dimension at most $p.$ 
	\par Now note that the derived subgroup of $\phi_{7}(1^5)$ is $\langle \alpha_2 \rangle \times \langle \alpha_3 \rangle$ and its center is $\langle \alpha_3 \rangle$. Consider the abelian normal subgroup $N_2 = \langle \alpha_1, \alpha_2, \alpha_3 \rangle$ of $\phi_{7}(1^5)$. Let $\chi \in \irr(N_2)$ such that $\chi(\alpha_1) = \zeta^{i_1}, \chi(\alpha_2) = \zeta^{i_2}$ and $\chi(\alpha_3) = \zeta^{i_3}$, where $\zeta$ is a primitive $p$-th root of unity and $0 \leq i_1, i_2, i_3 \leq (p-1).$ Assume that for some $0 \leq i, j \leq p-1,$ $\alpha^i \beta^j$ stabilizes $\chi.$ Let $\alpha_{1}^{a}\alpha_{2}^{b}\alpha_{3}^{c}\in N_{2}$. 
	Since the group $\phi_{7}(1^5)$ is of nilpotency class $3$, \begin{eqnarray*}
		\alpha^i (\beta^j \alpha_1^a) \alpha_2^b \alpha_3^c \beta^{-j} \alpha^{-i} & = &
		\alpha^i \alpha_1^a \beta^j \alpha_2^b \beta^{-j} \alpha^{-i} \alpha_3^{c-aj} \\
		& = & \alpha_2^{-ai} \alpha_1^a \alpha^i \beta^j \alpha_2^b \beta^{-j} \alpha^{-i} \alpha_3^{c-aj-a\binom{i}{2}} \\
		& = & \alpha_1^a \alpha_2^{b-ai} \alpha_3^{-ib+c-aj-a\binom{i}{2}} =\alpha_{1}^{a}\alpha_{2}^{b}\alpha_{3}^{c}\alpha_{2}^{-ai}\alpha_{3}^{-ib-aj-a\binom{i}{2}}. 
	\end{eqnarray*}
	
	
	Thus, $\alpha^{i}\beta^{j}$ stabilizes $\chi$ if, and only if, $\zeta^{-aii_2-(ib+aj+a\binom{i}{2})i_3} = 1.$ When $i_2 = 0$ and $i_3 \neq 0,$ then for $a=0$ and $b = \frac{-1}{i_3},$ we have $\zeta^{i}=1;$ which implies that $i=0.$ Further, $a=\frac{-1}{i_3}$ (note that $i_{3}$ is invertible modulo $p$) gives $\zeta^{j} =1$. Hence $j=0$ and it follows that the inertia group of $\chi$ in $\phi_{7}(1^{5})$ is $N_2.$ Therefore, the character of $\phi_7(1^5)$ induced from $\chi$ is irreducible of degree $p^2.$ Since $\phi_3(1^5)$ has no irreducible representations of dimension $p^2,$ the complex group algebras of $\phi_3(1^5)$ and $\phi_7(1^5)$ are not isomorphic. 
\end{proof}

\begin{proof}[{\bf Proof of Theorem~\ref{thm:s(G)=3}}] From Theorem~\ref{sumana}, it is clear that every $p$-group with $s(G) = 3$ has order $p^n$ where $n \in \{4, \ldots,  7\}$.
  In the following, we separately consider the cases of $n$ with $4 \leq n \leq 7$:

  \subsection*{n=4} For an odd prime $p,$ the only groups of order $p^4$ whose generalized corank is $3$ are $\phi_2(22), \phi_3(211)a, \phi_3(211)b_1$ and $\phi_3(211)b_{r_p}.$ It follows from the presentation that the derived subgroup of the group $\phi_2(22)$ is $\langle [\alpha_1, \alpha] \rangle,$ which is of order $p;$
  whereas for the rest of the three groups, the derived subgroup is of order $p^2.$ Therefore, each of these groups has $p^2$ and $(p^2-1)$ irreducible characters of degree $1$ and $p,$ respectively, and hence their complex group algebras are isomorphic. Since these groups are not of central type, it follows from Lemma \ref{centraltype} that they lie in the same $\mathbb{C}$-twist isomorphism class.

  On the other hand, when $p = 2$; the only two $2$-groups with $s(G) = 3$ that correspond to the case of $n=4$ are $C_{4} \rtimes C_{4}$ and $D_{16}.$
  Since the complex group algebras of these two groups are not isomorphic, each of these groups constitutes a singleton $\mathbb{C}$-twist isomorphism class.

  \subsection*{n=5}
  When $p$ is an odd prime, it follows from Theorem \ref{sumana} that the only groups of order $p^5$ with $s(G) = 3$ are $\phi_3(1^5)$, $\phi_7(1^5), H_1$ and $H_2$.  Consequently, the result follows from Proposition \ref{finalprop}. The derived subgroup of the groups $\phi_3(1^5)$ and $\phi_7(1^5)$ is elementary abelian of order $p^2$ and of $H_1$ and $H_2$ is of order $p.$ Thus the groups $\phi_3(1^5)$ and $\phi_7(1^5)$ have $p^3$ linear characters; where as $H_1$ and $H_2$ have $p^4$ linear characetrs. Therefore, no group in the set $\{\phi_3(1^5), \phi_7(1^5)\}$ can be $\mathbb{C}$-twist isomorphic to any group in $\{H_1, H_2\}.$ Now it follows from Lemma \ref{lastlem} and Proposition \ref{finalprop} that in this case the only non-singleton $\mathbb{C}$-twist isomorphism class is constituted by $H_1$ and $H_2.$

  For $p = 2$, the only two $2$-groups of order $32$ that have generalized corank $3$ are $C_{2}^{4} \rtimes C_{2}$ and $C_{2} \times ((C_{4}\times C_{2}) \rtimes C_{2})$. It can be checked using GAP~\cite{GAP4} that the size of the derived subgroup of the first group is $4$ and that of the second is $2.$ Thus the complex group algebras of these groups are not isomorphic, which establishes the desired result.

  \subsection*{n=6} The groups of order $p^6$ with $s(G) = 3$ are $\phi_{11}(1^6), \phi_{12}(1^6), \phi_{13}(1^6)$ and $\phi_{15}(1^6).$ Note that the size of the commutator subgroup of $\phi_{11}(1^6)$ is $p^3$ and of the rest of the groups is $p^2.$

  It can be checked that $\phi_{12}(1^6) = ES_p(p^3) \times ES_p(p^3)$ has $2p^2(p-1)$ inequivalent irreducible representations of dimension $p$ and $(p-1)^2$ of dimension $p^2.$ Further, it follows from \cite[Table~4.1]{james} that $\phi_{13}(1^6)$ has $(p^3-p^2)$ representations of dimension $p$; whereas $\phi_{15}(1^6)$ has no representation of dimension $p$. Thus each group of order $p^6$ with generalized corank 3 constitutes a singleton $\mathbb{C}$-twist isomorphism class.

  \subsection*{n=7} Any group of order $p^7$ with $s(G) = 3$ is isomorphic to $C_{p}^{(4)} \rtimes C_{p} \times C_{p}^{2}$. Hence such a group forms a singleton $\mathbb C$-twist isomorphism class.
\end{proof}




\subsection*{Acknowledgement} GK acknowledges the research support of the National Board for Higher Mathematics, Department of Atomic Energy, Govt. of India (0204/16(7)/2022/R\&D-II/ 11978). SK acknowledges the support of the Council of Scientific and
Industrial Research (CSIR), India (09/092(1066)/2020-EMR-I). PS thanks the support of SERB power grant(SPG/2022/001099). 
	
\subsection*{Declaration of interests} The authors declare that they have no known competing financial interests or personal relationships that could have appeared to influence the work reported in this paper.

\bibliographystyle{amsplain}
\bibliography{TP}
\end{document}